\documentclass[12pt]{amsart}

\usepackage{amsmath,amssymb,mathrsfs,amscd,enumerate,color,url}
\usepackage{hyperref}

\theoremstyle{definition}

\newtheorem{theorem}{Theorem}[section]

\newtheorem{definition}{Definition}[section]

\newtheorem{corollary}{Corollary}[section]
\newtheorem{proposition}{Proposition}[section]
\newtheorem{example}{Example}[section]
\newtheorem{remark}{Remark}[section]

\numberwithin{equation}{section}

\begin{document}

\title{On a new generalization of metric spaces}

\author{Mohamed Jleli, Bessem Samet}



\maketitle

\begin{abstract}
In this paper, we introduce the $\mathcal{F}$-metric space concept, which generalizes the metric space notion. We define a natural topology $\tau_{\mathcal{F}}$ in such spaces and we study their topological properties. Moreover, we establish a new version of the Banach contraction principle in the setting of $\mathcal{F}$-metric spaces. Several examples are presented to illustrate our study.
\end{abstract}

\section{Introduction}\label{sec1}

A metric on a nonempty set $X$ is a mapping $d: X\times X\to [0,+\infty)$ satisfying the following properties:
\begin{itemize}
\item[(i)] $d(x,y)=0$ if and only if $x=y$.
\item[(ii)] $d(x,y)=d(y,x)$.
\item[(iii)] $d(x,y)\leq d(x,z)+d(z,y)$.
\end{itemize}
If $d$ is a metric on $X$, then the pair $(X,d)$ is said to be a metric space.   The theory of metric spaces is the general theory which underlies several branches of mathematical analysis, as real analysis, complex analysis, multidimensional calculus, etc.

In recent years, many interesting generalizations (or extensions)  of the metric space
concept appeared.  Czerwik \cite{C} introduced the notion of a $b$-metric.  Khamsi and Hussain  \cite{KH} reintroduced this notion  under the name metric-type.  In \cite{FA},  Fagin et al.  introduced the notion of $s$-relaxed$_p$ metric. Note that any $s$-relaxed$_p$ metric is a $b$-metric, but the converse is not true in general (see \cite{KS}). 
G\"ahler \cite{GA} introduced the notion of a 2-metric, which is a mapping defined on the product set $X\times X\times X$, and satisfying certain conditions. G\"ahler claimed that a 2-metric is a generalization of the usual notion of a metric. 
However, different authors showed that no relations between these two concepts exist (see, for example \cite{H}). A more appropriate notion of generalized metric space was introduced by Mustafa and Sims \cite{MS} under the name G-metric space.
In \cite{B}, Branciari suggested a new generalization of the metric notion by replacing the triangle inequality (iii) by a more general one involving four points. Matthews \cite{M} introduced the notion of a partial metric as a part of the study of denotational semantics of dataflow networks. Recently, we introduced \cite{JS} the concept of JS-metric, where the triangle inequality is replaced by a $\limsup$-condition. For more details about the above cited concepts
and other generalizations of the metric notion, we refer the reader to the nice book \cite{KS} by  Kirk and Shahzad.

In this paper, we introduce a new generalization of the metric space notion, which we call an $\mathcal{F}$-metric space. We compare our concept with existing generalizations from the literature. Next, we define a natural topology $\tau_{\mathcal{F}}$ on these spaces, and we study their  topological properties. Moreover, a new version of the Banach contraction principle is established in the setting of $\mathcal{F}$-metric spaces.

The paper is organized as follows. In Section \ref{sec2}, the concept of $\mathcal{F}$-metric spaces is introduced.  We show that any metric space is an $\mathcal{F}$-metric space but the converse is not true in general, which confirms that our concept is more general than the standard metric concept.  Moreover, we compare our proposed notion with previous generalizations of metric spaces. More precisely, we show that any $s$-relaxed$_p$-metric space is an $\mathcal{F}$-metric space (see Example \ref{JL}). Further, we provide an example of an  $\mathcal{F}$-metric space that cannot be an 
$s$-relaxed$_p$-metric space (see Example \ref{EXJ}), which confirms that the class of  $\mathcal{F}$-metric spaces is more large than the class of $s$-relaxed$_p$-metric spaces. A comparison with b-metric spaces is also considered. We show that there exist $\mathcal{F}$-metric spaces that are not b-metric spaces (see Example \ref{JL}) and there exist b-metric spaces that are not $\mathcal{F}$-metric spaces (see Proposition \ref{T0}). In Section \ref{sec3}, we introduce the notion of $\mathcal{F}$-metric boundedness, which is used to provide a characterization of $\mathcal{F}$-metrics (see  Theorem \ref{T1}). In Section \ref{sec4}, a topology $\tau_{\mathcal{F}}$ is introduced on $\mathcal{F}$-metric spaces using the concept of balls. It is well-known that in standard metric spaces, the closed ball is closed with respect to the topology defined via balls (or equivalently the sequential topology). In our situation, we do not know whether closed balls are closed with respect to $\tau_{\mathcal{F}}$. However, we provide a sufficient condition so that any closed ball  is closed with respect to $\tau_{\mathcal{F}}$ (see Proposition \ref{CL}).   Additional topological properties are also discussed, as compactness, completeness, etc. In Section \ref{sec5}, we establish the Banach contraction principle in the setting of $\mathcal{F}$-metric spaces (see Theorem \ref{TB}).

\section{A generalized metric space}\label{sec2}

Let $\mathcal{F}$ be the set of functions $f:(0,+\infty)\to \mathbb{R}$   satisfying the following conditions:
\begin{itemize}
\item[($\mathcal{F}_1$)] $f$ is non-decreasing, i.e.,
$0<s<t\implies f(s)\leq f(t)$.
 \item[($\mathcal{F}_2$)] For every sequence $\{t_n\}\subset (0,+\infty)$, we have
 $$
 \lim_{n\to +\infty} t_n=0 \Longleftrightarrow \lim_{n\to +\infty}f(t_n)=-\infty.
 $$ 
\end{itemize}

We  generalize the concept of metric spaces as follows.

\begin{definition}
Let $X$ be a nonempty set, and let $D: X\times X\to [0,+\infty)$ be a given mapping. Suppose that there exists $(f,\alpha)\in \mathcal{F}\times [0,+\infty)$ such that 
\begin{itemize}
\item[(D1)] $(x,y)\in X\times X, \, D(x,y)=0\Longleftrightarrow x=y$.
\item[(D2)] $D(x,y)=D(y,x)$, for all $(x,y)\in X\times X$.
\item[(D3)] For every $(x,y)\in X\times X$, for every $N\in \mathbb{N}$, $N\geq 2$, and for every $\displaystyle (u_i)_{i=1}^N\subset X$ with $(u_1,u_N)=(x,y)$, we have
$$
D(x,y)>0\implies f(D(x,y))\leq f \left(\sum_{i=1}^{N-1} D(u_i,u_{i+1})\right)+\alpha.
$$
\end{itemize}
Then $D$ is said to be an $\mathcal{F}$-metric on $X$, and the pair $(X,D)$ is said to be an $\mathcal{F}$-metric space.
\end{definition}

Observe that any  metric on $X$ is an $\mathcal{F}$-metric on $X$. Indeed, if $d$ is a metric on $X$, then it satisfies (D1) and (D2). On the other hand, by the triangle inequality, for every $(x,y)\in X\times X$, for every $N\in \mathbb{N}$, $N\geq 2$, and for every $\displaystyle (u_i)_{i=1}^N\subset X$ with $(u_1,u_N)=(x,y)$, we have
$$
d(x,y)\leq \sum_{i=1}^{N-1} d(u_i,u_{i+1}),
$$
which yields, 
$$
d(x,y)>0 \implies \ln\left(d(x,y)\right)\leq \ln\left(\sum_{i=1}^{N-1} d(u_i,u_{i+1})\right).
$$
Then $d$ satisfies  (D3) with  $f(t)=\ln t$, $t>0$, and $\alpha=0$.

In the following, some examples of $\mathcal{F}$-metric spaces which are not metric spaces are presented.

\begin{example}\label{exr}
Let $X=\mathbb{N}$, and let $D: X\times X\to [0,+\infty)$ be the mapping defined by
\begin{eqnarray}\label{DD}
D(x,y)=\left\{\begin{array}{lll}
(x-y)^2, &\mbox{ if }& (x,y)\in [0,3]\times [0,3],\\ \\
|x-y|, &\mbox{ if }& (x,y)\not\in [0,3]\times [0,3],
\end{array}
\right.
\end{eqnarray}
for all $(x,y)\in X\times X$. It can be easily seen that $D$ satisfies (D1) and (D2). However, $D$ doesn't satisfy the triangle inequality, since
$$
d(1,3)=4>1+1=d(1,2)+d(2,3).
$$
Hence, $D$ is not a metric on $X$. Further, let us fix a certain $(x,y)\in X\times X$ such that $D(x,y)>0$. Let $\displaystyle (u_i)_{i=1}^N\subset X$, where $N\in \mathbb{N}$, $N\geq 2$, and $(u_1,u_N)=(x,y)$. Let
$$
I=\{i=1,2,\cdots,N-1:\, (u_i,u_{i+1})\in [0,3]\times [0,3]\}
$$
and
$$
J=\{1,2,\cdots,N-1\}\backslash I.
$$
Therefore, we have
\begin{eqnarray*}
\sum_{i=1}^{N-1} D(u_i,u_{i+1})&=&
\sum_{i\in I} D(u_i,u_{i+1})+\sum_{j\in J} D(u_j,u_{j+1})\\
&= & \sum_{i\in I} (u_{i+1}-u_{i})^2 +\sum_{j\in J} |u_{j+1}-u_j|.
\end{eqnarray*}
Next, we discuss two possible cases.\\
Case 1: If $(x,y)\not\in [0,3]\times [0,3]$. In this case, we have
\begin{eqnarray*}
D(x,y) &=& |x-y|\\
&\leq &\sum_{i=1}^{N-1} |u_{i+1}-u_{i}|\\
&= &\sum_{i\in I} |u_{i+1}-u_{i}|+\sum_{j\in I} |u_{j+1}-u_{j}|.
\end{eqnarray*}
On the other hand, observe that
$$
|u_{i+1}-u_{i}|\leq (u_{i+1}-u_{i})^2,\quad i\in I.
$$
Therefore, we deduce that 
\begin{eqnarray*}
D(x,y) &\leq & \sum_{i\in I} (u_{i+1}-u_{i})^2 +\sum_{j\in J} |u_{j+1}-u_j|\\
&=& \sum_{i=1}^{N-1} D(u_i,u_{i+1}).
\end{eqnarray*}
Case 2: If $(x,y)\in [0,3]\times [0,3]$.  In this case, we have
\begin{eqnarray*}
D(x,y)&=&|x-y|^2\\
&\leq &  3|x-y|\\
&\leq & 3 \left(\sum_{i\in I} |u_{i+1}-u_{i}|+\sum_{j\in J} |u_{j+1}-u_{j}|\right)\\
&\leq & 3\left(\sum_{i\in I} |u_{i+1}-u_{i}|^2+\sum_{i\in J} |u_{j+1}-u_{j}|\right)\\
&=&3  \sum_{i=1}^{N-1} D(u_i,u_{i+1}).
\end{eqnarray*}
Next, combining the above cases, we deduce that for every $(x,y)\in X\times X$, for every $N\in \mathbb{N}$, $N\geq 2$, and for every $\displaystyle (u_i)_{i=1}^N\subset X$ with $(u_1,u_N)=(x,y)$, we have
\begin{equation}\label{RC}
D(x,y)>0 \implies D(x,y)\leq 3  \sum_{i=1}^{N-1} D(u_i,u_{i+1}),
\end{equation}
which yields
$$
D(x,y)>0 \implies  \ln(D(x,y))\leq \ln\left(\sum_{i=1}^{N-1} D(u_i,u_{i+1})\right)+\ln 3.
$$
This proves that $D$ satisfies (D3) with $f(t)=\ln t$, $t>0$, and $\alpha=\ln 3$. Then  $D$ is an $\mathcal{F}$-metric on $X$. 
\end{example}

\begin{example}[The class of $s$-relaxed$_p$ metrics] \label{JL}
Let $d: X\times X\to [0,+\infty)$ be an $s$-relaxed$_p$ metric on $X$ (see \cite{FA}), i.e., $d$ satisfies (D1), (D2), and 
\begin{itemize}
\item[(S)] There exists $K\geq 1$ such that for every $(x,y)\in X\times X$, for every $N\in \mathbb{N}$, $N\geq 2$, and for every $\displaystyle (u_i)_{i=1}^N\subset X$ with $(u_1,u_N)=(x,y)$, we have
$$
d(x,y)\leq K \sum_{i=1}^{N-1} d(u_i,u_{i+1}).
$$
\end{itemize}
Then $d$ satisfies (D3) with $f(t)=\ln t$, $t>0$, and $\alpha=\ln K$. As consequence, any $s$-relaxed$_p$ metric on $X$ is an $\mathcal{F}$-metric on $X$.
\end{example}

\begin{remark}
Note that from \eqref{RC},  the mapping $D$ defined by \eqref{DD}  is an  $s$-relaxed$_p$ metric on $X$ with $K=3$.
\end{remark}

\begin{example}[The class of bounded 2-metric spaces]
Let $\sigma: X\times X\times X\to [0,+\infty)$ be a mapping satisfying the following conditions:
\begin{itemize}
\item[($\sigma_1$)] $(a,b)\in X\times X,\, a\neq b \implies \exists\,c\in X:\, \sigma(a,b,c)\neq 0$.
\item[($\sigma_2$)] For all $(a,b,c)\in X\times X\times X$, 
$\sigma(a,b,c)=0$ if and only if at least two elements from $\{a,b,c\}$ are equal.
\item[($\sigma_3$)] $(a,b,c)\in X\times X\times X\implies 
\sigma(a,b,c)=\sigma(u,v,w)$, where $\{u,v,w\}$ is any permutation of $\{a,b,c\}$.
\item[($\sigma_4$)] For all $(a,b,c)\in X\times X\times X$, we have
$$
\sigma(a,b,c)\leq \sigma(a,b,d)+\sigma(b,c,d)+\sigma(c,a,d).
$$
\end{itemize}
Then $\sigma$ is called a 2-metric on $X$, and $(X,\sigma)$ is called a 2-metric space (see \cite{GA}). Moreover, suppose that $\displaystyle\sup_{x,y,z\in X} \sigma(x,y,z)<+\infty$. In this case,  $(X,\sigma)$ is said to be a bounded 2-metric space. Define the mapping $D_\sigma: X\times X\to [0,+\infty)$ by
$$
D_\sigma(x,y)=\sup_{a\in X}\sigma(a,x,y),\quad (x,y)\in X\times X.
$$
It was proved in \cite{AN} that $D_\sigma$ is an s-relaxed$_p$ metric on $X$ with $K=2$. Therefore, $D_\sigma$ is an $\mathcal{F}$-metric on $X$.
\end{example}

The next example shows that the class of $\mathcal{F}$-metrics is more large than the class of $s$-relaxed$_p$ metrics.

\begin{example}\label{EXJ}
Let $X=\mathbb{N}$, and let $D: X\times X\to [0,+\infty)$ be the mapping defined by
\begin{eqnarray}\label{DDD}
D(x,y)=\left\{\begin{array}{lll}
\exp\left(|x-y|\right), &\mbox{ if }& x\neq y,\\ \\
0, &\mbox{ if }& x=y,
\end{array}
\right.
\end{eqnarray}
for all $(x,y)\in X\times X$. It can be easily seen that $D$ satisfies (D1) and (D2).

First, let us prove that $D$ cannot be  an $s$-relaxed$_p$ metric. We argue by contradiction, by supposing that $D$ satisfies the condition (S) of Example \ref{JL} with a certain $K\geq 1$. Therefore, we have
$$
D(2n,0)\leq K\left(D(2n,n)+D(n,0)\right),\quad n\in \mathbb{N}^*,
$$
that is,
$$
\exp(n)\leq 2K,\quad n\in \mathbb{N}^*.
$$
Passing to the limit as $n\to +\infty$, we obtain a contradiction. Therefore, $D$ is not an $s$-relaxed$_p$ metric.

 Next, we shall prove that $D$ belongs to the class of $\mathcal{F}$-metrics.  Let 
$$
f(t)=\frac{-1}{t},\quad t>0.
$$
It can be easily seen that $f\in \mathcal{F}$.   In order to check (D3), let us fix $(x,y)\in X\times X$ with $D(x,y)>0$. For every $N\in \mathbb{N}$, $N\geq 2$, and for every $\displaystyle (u_i)_{i=1}^N\subset X$ with $(u_1,u_N)=(x,y)$, we have
\begin{align*}
&1+f\left(\sum_{i=1}^{N-1} D(u_i,u_{i+1})\right)
-f(D(x,y))\\
&=1-\frac{1}{\displaystyle\sum_{i=1:N-1,\,u_{i+1}\neq u_i }\exp\left(|u_{i+1}-u_i|\right)}+\frac{1}{\exp\left(|x-y|\right)}\\
&\geq 1-1+ \frac{1}{\exp\left(|x-y|\right)}\\
&\geq 0.
\end{align*}
Therefore, we have
$$
f(D(x,y))\leq f\left(\sum_{i=1}^{N-1} D(u_i,u_{i+1})\right)+1.
$$
This proves that  $D$ satisfies (D3) with $f(t)=\frac{-1}{t}$, $t>0$, and $\alpha= 1$. Then $D$ is an $\mathcal{F}$-metric.
\end{example}

It was proved in \cite{FA} (see also \cite{KS}) that there is a b-metric  space that is not an s-relaxed$_p$ metric space for any $K\geq 1$. We shall prove an analogous result  for the case of $\mathcal{F}$-metric spaces. First, recall that a mapping $d: X\times X\to [0,+\infty)$ is said to be a b-metric on $X$ if it satisfies (D1), (D2), and  
\begin{itemize}
\item[(S)'] There exists $K\geq 1$ such that 
$$
d(x,y)\leq K\left(d(x,z)+d(z,y)\right),\quad (x,y,z)\in X\times X\times X.
$$
\end{itemize}
Observe that (S)$\implies$ (S)'. Therefore, any s-relaxed$_p$ metric is a b-metric. However, as we mentioned before, the converse is not true in general.

\begin{proposition}\label{T0}
There is a b-metric space that is not an $\mathcal{F}$-metric space.
\end{proposition}

\begin{proof}
Let $X=[0,1]$, and let $d: X\times X\to [0,+\infty)$ be the mapping defined by
$$
d(x,y)=(x-y)^2,\quad (x,y)\in X\times X.
$$ 
It can be easily seen (see, for example \cite{KS}) that $d$ is a b-metric on $X$ with constant $K=2$.  Suppose that there exists $(f,\alpha)\in \mathcal{F}\times [0,+\infty)$ such that $d$ satisfies (D3). Let $n\in \mathbb{N}^*$,   and let 
$$
u_i=\frac{i}{n},\quad i=0,2,\cdots,n.
$$ 
By (D3), we have
$$
f(d(0,1))\leq f\left(d(0,u_1)+d(u_1,u_2)+\cdots+d(u_{n-1},1)\right)+\alpha,\quad n\in \mathbb{N}^*,
$$
i.e.,
$$
f(1)\leq f\left(\frac{1}{n}\right)+\alpha,
\quad n\in \mathbb{N}^*.
$$
On the other hand, by ($\mathcal{F}_2$), we have
$$
\lim_{n\to +\infty} f\left(\frac{1}{n}\right)+\alpha=-\infty,
$$
which is a contradiction. 
\end{proof}

\begin{remark}
We proved in  Example \ref{EXJ} that the mapping $D$ defined 
by \eqref{DDD} is an $\mathcal{F}$-metric on $X$ but it is not an $s$-relaxed$_p$ metric. It can be easily seen that $D$ is not also a b-metric on $X$. 
\end{remark}

\section{Characterization of $\mathcal{F}$-metrics}\label{sec3}

In this section, we introduce the concept of $\mathcal{F}$-metric boundedness, which will be used later to give a characterization  of $\mathcal{F}$-metrics.

\begin{definition}
Let $X$ be a nonempty set, and let $D: X\times X\to [0,+\infty)$ be a given mapping satisfying (D1) and (D2). We say that the pair $(X,D)$ is $\mathcal{F}$-metric bounded with respect to   $(f,\alpha)\in \mathcal{F}\times [0,+\infty)$, if there exists a metric $d$ on $X$ such that 
\begin{equation}\label{MB}
(x,y)\in X\times X,\, D(x,y)>0\implies f(d(x,y))\leq f(D(x,y))\leq f(d(x,y))+\alpha.
\end{equation}
\end{definition}

We have the following result.

\begin{theorem}\label{T1}
Let $X$ be a nonempty set, and let $D: X\times X\to [0,+\infty)$ be a given mapping satisfying (D1) and (D2). Let $(f,\alpha)\in \mathcal{F}\times [0,+\infty)$, and suppose that $f$ is continuous from the right. Then the following statements are equivalent:
\begin{itemize}
\item[(i)] $(X,D)$ is an $\mathcal{F}$-metric on $X$ with $(f,\alpha)$ defined above.
\item[(ii)] $(X,D)$ is $\mathcal{F}$-metric bounded with respect to   $(f,\alpha)$.
\end{itemize}
\end{theorem}

\newpage

\begin{proof}
(i)$\implies$ (ii):  Assume that $(X,D)$ is  an $\mathcal{F}$-metric on $X$ with respect to  $(f,\alpha)$. Let us define the mapping $d: X\times X\to [0,+\infty)$  by
$$
d(x,y)=\inf\left\{\sum_{i=1}^{N-1} D(u_i,u_{i+1}):\, N\in \mathbb{N}, \, N\geq 2, \, (u_i)_{i=1}^N \subset X,\, (u_1,u_N)=(x,y)\right\},
$$
for all $(x,y)\in X\times X$. We shall prove that $d$ is a metric on $X$. Since $D(x,x)=0$, for all $x\in X$, it follows from the definition of $d$ that 
$$
d(x,x)=0,\quad x\in X.
$$
Now, let $(x,y)\in X\times X$ be such that $x\neq y$. Suppose that $d(x,y)=0$. Let $\varepsilon>0$, by the definition of $d$, there exist $N\in \mathbb{N}$, $N\geq 2$, and $(u_i)_{i=1}^N \subset X$ with $(u_1,u_N)=(x,y)$ such that
$$
\sum_{i=1}^{N-1} D(u_i,u_{i+1}) <\varepsilon.
$$
By ($\mathcal{F}_1$), we obtain
\begin{equation}\label{in1}
f\left(\sum_{i=1}^{N-1} D(u_i,u_{i+1})\right)\leq f(\varepsilon).
\end{equation}
On the other hand, by (D3), we have
\begin{equation}\label{in2}
f(D(x,y))\leq f\left(\sum_{i=1}^{N-1} D(u_i,u_{i+1})\right)+\alpha.
\end{equation}
Using \eqref{in1} and \eqref{in2}, we obtain
$$
f(D(x,y))\leq f(\varepsilon)+\alpha,\quad \varepsilon>0.
$$
But, using ($\mathcal{F}_2$), we have
$$
\lim_{\varepsilon\to 0^+}\left(f(\varepsilon)+\alpha\right)
=-\infty,
$$
which is a contradiction. Therefore, we have $d(x,y)>0$. From the definition of $d$ and (D2), it can be easily seen that $d(x,y)=d(y,x)$, for all $(x,y)\in X\times X$. In order to check the triangle inequality, let $x,y$ and $z$ be three given points in $X$, and let $\rho>0$. By the definition of $d$, there exist two chains of points $x=u_1,u_2,\cdots,u_n=y$ and $y=u_n,u_{n+1},\cdots,u_m=z$ such that 
$$
\sum_{i=1}^{n-1} D(u_i,u_{i+1})<d(x,y)+\rho
$$
and
$$
\sum_{i=n}^{m-1} D(u_i,u_{i+1})<d(y,z)+\rho.
$$
Adding the above inequalities, we obtain
$$
d(x,z)\leq \sum_{i=1}^{m-1}D(u_i,u_{i+1})<d(x,y)+d(y,z)+2\rho,\quad \rho>0.
$$
Passing to the limit as $\rho\to 0^+$, we get
$$
d(x,z)\leq d(x,y)+d(y,z).
$$
As consequence, we deduce that $d$ is a metric on $X$. Next, we shall prove that $d$ satisfies \eqref{MB}. Let $(x,y)\in X\times X$ be such that $D(x,y)>0$.  From the definition of $d$, it is clear that
$$
d(x,y)\leq D(x,y),
$$
which implies from ($\mathcal{F}_1$) that 
\begin{equation}\label{IK1}
f(d(x,y))\leq f(D(x,y)).
\end{equation}
Let $\varepsilon>0$. By the definition of $d$, there exist $N\in \mathbb{N}$, $N\geq 2$, and $(u_i)_{i=1}^N \subset X$ with $(u_1,u_N)=(x,y)$ such that
$$
\sum_{i=1}^{N-1} D(u_i,u_{i+1}) <d(x,y)+\varepsilon.
$$
By ($\mathcal{F}_1$), we obtain
$$
f\left(\sum_{i=1}^{N-1} D(u_i,u_{i+1})\right)\leq f(d(x,y)+\varepsilon).
$$
Using (D3) and the above inequality, we get
$$
f(D(x,y))\leq f(d(x,y)+\varepsilon)+\alpha,\quad \varepsilon>0.
$$
Passing to the limit as $\varepsilon\to 0^+$, and using the right continuity of $f$, we obtain
\begin{equation}\label{IK2}
f(D(x,y))\leq f(d(x,y))+\alpha.
\end{equation}
By \eqref{IK1} and \eqref{IK2}, we have
$$
f(d(x,y))\leq f(D(x,y))\leq f(d(x,y))+\alpha.
$$
Then \eqref{MB} is satisfied and $(X,D)$ is $\mathcal{F}$-metric bounded with respect to $(f,\alpha)$.\\
(ii)$\implies $(i): Suppose that $(X,D)$ is $\mathcal{F}$-metric bounded with respect to $(f,\alpha)$, that is, there exists a certain metric $d$ on $X$ such that \eqref{MB} is satisfied. We have just to prove that $D$ satisfies (D3). Let $(x,y)\in X\times X$ be such that $D(x,y)>0$. Let $N\in \mathbb{N}$, $N\geq 2$, and $(u_i)_{i=1}^N \subset X$ with $(u_1,u_N)=(x,y)$. Since $d$ is a metric on $X$, the triangle inequality yields
\begin{equation}\label{gd}
d(x,y)\leq \sum_{i=1}^{N-1} d(u_i,u_{i+1}).
\end{equation}
On the other hand, using ($\mathcal{F}_1$) and the fact that 
$$
(u,v)\in X\times X,\, D(u,v)>0\implies f(d(u,v))\leq f(D(u,v)), 
$$
we deduce that 
\begin{equation}\label{gd2}
d(u,v)\leq D(u,v),\quad (u,v)\in X\times X.
\end{equation}
By \eqref{gd} and \eqref{gd2}, we obtain
$$
d(x,y)\leq \sum_{i=1}^{N-1} D(u_i,u_{i+1}),
$$
which implies by ($\mathcal{F}_1$) that 
$$
f(d(x,y))+\alpha\leq f\left(\sum_{i=1}^{N-1} D(u_i,u_{i+1})\right)+\alpha.
$$
Using the above inequality and the fact that 
$$
f(D(x,y))\leq f(d(x,y))+\alpha,
$$
we deduce that 
$$
f(D(x,y))\leq f\left(\sum_{i=1}^{N-1} D(u_i,u_{i+1})\right)+\alpha.
$$
Therefore, (D3) is satisfied and $(X,D)$ is an $\mathcal{F}$-metric on $X$.
\end{proof}

\begin{remark}
Observe that from the proof of Theorem \ref{T1}, the right continuity assumption imposed on $f$ is used only to prove that (i)$\implies$ (ii). However, for any $f\in \mathcal{F}$, we have (ii)$\implies$ (i).
\end{remark}

\section{Topological $\mathcal{F}$-metric spaces}\label{sec4}

In this section, we discuss  a natural topology defined on  $\mathcal{F}$-metric spaces.

\begin{definition}
Let $(X,D)$ be an $\mathcal{F}$-metric space. A subset $\mathcal{O}$ of $X$ is said to be $\mathcal{F}$-open if for every $x\in \mathcal{O}$, there is some $r>0$ such that $B(x,r)\subset \mathcal{O}$, where
$$
B(x,r)=\{y\in X:\, D(x,y)<r\}.
$$
We say that a subset $\mathcal{C}$ of $X$ is $\mathcal{F}$-closed if $X\backslash \mathcal{C}$ is $\mathcal{F}$-open.
We denote by $\tau_{\mathcal{F}}$ the family of all $\mathcal{F}$-open subsets of $X$.
\end{definition}

The following result can be proved easily.

\begin{proposition}
Let $(X,D)$ be an $\mathcal{F}$-metric space. Then $\tau_{\mathcal F}$ is a topology on $X$.
\end{proposition}

\begin{proposition}\label{IH}
Let $(X,D)$ be an $\mathcal{F}$-metric space. Then, for any nonempty subset $A$ of $X$, the following statements are equivalent:
\begin{itemize}
\item[(i)] $A$ is $\mathcal{F}$-closed.
\item[(ii)] For any sequence $\{x_n\}\subset A$, we have
$$
\lim_{n\to +\infty}D(x_n,x)=0,\,x\in X\implies x\in A.
$$
\end{itemize}
\end{proposition}

\begin{proof}
Assume that $A$ is $\mathcal{F}$-closed, and let $\{x_n\}$ be a sequence in $A$ such that 
\begin{equation}\label{LS}
\lim_{n\to +\infty}D(x_n,x)=0,
\end{equation}
where $x\in X$. Suppose that $x\in X\backslash A$. Since $A$ is $\mathcal{F}$-closed, $X\backslash A$ is $\mathcal{F}$-open. Therefore, there exists some $r>0$ such that $B(x,r)\subset X\backslash A$, i.e. $B(x,r)\cap A=\emptyset$. On the other hand,  by \eqref{LS}, there exists some $N\in \mathbb{N}$ such that 
$$
D(x_n,x)<r,\quad n\geq N,
$$
i.e.
$$
x_n\in B(x,r),\quad n\geq N.
$$ 
Hence, $x_N\in B(x,r)\cap A$, which leads to a contradiction. Therefore, we deduce that $x\in A$, and (i)$\implies$(ii) is proved. Conversely, assume that (ii) is satisfied.  Let $x\in X\backslash A$. We have to prove that there is some $r>0$ such that $B(x,r)\subset X\backslash A$. We argue by contradiction by supposing that for every $r>0$, there exists $x_r\in B(x,r)\cap A$. This implies that for any $n\in \mathbb{N}^*$, there exists $x_n\in B(x,\frac{1}{n})\cap A$. 
Then 
$$
\{x_n\}\subset A,\, \lim_{n\to +\infty} D(x_n,x)=0.
$$
By (ii), this implies that $x\in A$, which is a contradiction with $x\in X\backslash A$. Hence, $A$ is $\mathcal{F}$-closed, and (ii) $\implies $(i).
\end{proof}

\begin{proposition}\label{CL}
Let $(X,D)$ be an $\mathcal{F}$-metric space, $a\in X$, and $r>0$. We denote by $\mathbf{B}(a,r)$ the subset of $X$ defined by
$$
\mathbf{B}(a,r)=\{x\in X:\, D(a,x)\leq r\}.
$$
Suppose that for every sequence $\{x_n\}\subset X$, we have
\begin{equation}\label{JAL}
\lim_{n\to +\infty} D(x_n,x)=0,\,x\in X \implies D(x,y)\leq \limsup_{n\to +\infty} D(x_n,y),\, y\in X.
\end{equation}
Then $\mathbf{B}(a,r)$ is $\mathcal{F}$-closed.
\end{proposition}

\begin{proof}
Let $\{x_n\}\subset \mathbf{B}(a,r)$ be a sequence such that
$$
\lim_{n\to +\infty} D(x_n,x)=0,
$$
for a certain $x\in X$. From proposition \ref{IH}, we have to prove that $x\in \mathbf{B}(a,r)$. By the definition of $\mathbf{B}(a,r)$, we have
$$
D(x_n,a)\leq r,\quad n\in \mathbb{N}.
$$
Passing to the supremum limit as $n\to +\infty$ and using \eqref{JAL}, we obtain
$$
D(x,a)\leq \limsup_{n\to +\infty} D(x_n,a)\leq r,
$$
which yields $x\in \mathbf{B}(a,r)$. Therefore, $\mathbf{B}(a,r)$ is $\mathcal{F}$-closed.
 \end{proof}

\begin{remark}
Proposition \ref{CL} provides only a sufficient condition ensuring that  $\mathbf{B}(a,r)$ is $\mathcal{F}$-closed. An interesting problem consists to find a necessary and sufficient condition under which $\mathbf{B}(a,r)$ is $\mathcal{F}$-closed.
\end{remark}

\begin{definition}
Let $(X,D)$ be an $\mathcal{F}$-metric space. Let $A$ be a nonempty subset of $X$. We denote by $\overline{A}$ the closure of $A$ with respect to the topology $\tau_{\mathcal{F}}$, i.e. $\overline{A}$ is  the intersection of all $\mathcal{F}$-closed subsets of $X$  containing $A$.  Clearly, $\overline{A}$ is the smallest $\mathcal{F}$-closed subset which contains $A$. 
\end{definition}

\begin{proposition}\label{ayham}
Let $(X,D)$ be an $\mathcal{F}$-metric space.  Then, for any nonempty subset $A$ of $X$, we have 
$$
x\in \overline{A},\, r>0\implies B(x,r)\cap A\neq \emptyset.
$$
\end{proposition}

\begin{proof}
Let $(f,\alpha)\in \mathcal{F}\times [0,+\infty)$ be such that (D3) is satisfied. Let us define the set
$$
A'=\{x\in X:\, \mbox{ for any } r>0, \mbox{ there exists } a\in A:\, D(x,a)<r\}.
$$
By (D1), it can be easily seen that $A\subset A'$. Next, we shall prove that $A'$ is $\mathcal{F}$-closed. Let $\{x_n\}$ be a sequence in $A'$ such that
\begin{equation}\label{eqA}
\lim_{n\to +\infty} D(x_n,x)=0,\quad x\in X.
\end{equation}
Let $r>0$. By ($\mathcal{F}_2$), there exists some $\delta_r>0$ such that
\begin{equation}\label{delta}
0<t<\delta_r\implies \mu(t)<\mu(r)-\alpha.
\end{equation}
On the other hand, by \eqref{eqA}, there exists some $N\in \mathbb{N}$ such that
$$
D(x_n,x)< \frac{\delta_r}{3},\quad n\geq N.
$$
Since $x_N\in A'$, there exists $a\in A$ such that 
$$
D(x_N,a)< \frac{\delta_r}{3}.
$$ 
If $D(x,a)>0$, by (D3), we have 
$$
f(D(x,a))\leq f(D(x,x_N)+D(x_N,a))+\alpha\leq f\left(\frac{2\delta_r}{3}\right)+\alpha.
$$
But by \eqref{delta}, since $\frac{2\delta_r}{3}<\delta_r$, we obtain
$$
f\left(\frac{2\delta_r}{3}\right)<f(r)-\alpha.
$$
Hence,
$$
f(D(x,a))<f(r),
$$
which implies from ($\mathcal{F}_1$) that $D(x,a)<r$. Therefore, in all cases, we have
$$
D(x,a)<r,
$$
which yields $x\in A'$. Then by Proposition \ref{IH}, $A'$ is $\mathcal{F}$-closed, which contains $A$. Then $\overline{A}\subset A'$, which yields the desired result.
\end{proof}

\begin{definition}
Let $(X,D)$ be an $\mathcal{F}$-metric space. Let $\{x_n\}$ be a sequence in $X$. We say that $\{x_n\}$ is $\mathcal{F}$-convergent to $x\in X$ if  $\{x_n\}$ is convergent to $x$ with respect to the topology $\tau_{\mathcal{F}}$, i.e. for every $\mathcal{F}$-open subset $\mathcal{O}_x$ of $X$ containing $x$, there exists some $N\in \mathbb{N}$ such that $x_n\in \mathcal{O}_x$, for all $n\geq N$. In this case, we say that $x$ is the limit of $\{x_n\}$.
\end{definition}

The following result follows immediately from the above definition and the definition of $\tau_{\mathcal{F}}$. 

\begin{proposition}
Let $(X,D)$ be an $\mathcal{F}$-metric space. Let $\{x_n\}$ be a sequence in $X$, and $x\in X$. The following statements are equivalent:
\begin{itemize}
\item[(i)] $\{x_n\}$ is $\mathcal{F}$-convergent to $x$.
\item[(ii)] $\displaystyle\lim_{n\to +\infty}D(x_n,x)=0$.
\end{itemize}
\end{proposition}

The next result shows that the limit of an  $\mathcal{F}$-convergent sequence is unique.

\begin{proposition}
Let $(X,D)$ be an $\mathcal{F}$-metric space. Let $\{x_n\}$ be a sequence in $X$. Then
$$
(x,y)\in X\times X,\, \lim_{n\to +\infty} D(x_n,x)=\lim_{n\to +\infty} D(x_n,y)=0\implies x=y.
$$
\end{proposition}

\begin{proof}
Let $(x,y)\in X\times X$ be such that 
$$
\lim_{n\to +\infty} D(x_n,x)=\lim_{n\to +\infty} D(x_n,y)=0.
$$
Suppose that $x\neq y$, i.e. (from (D1)) $D(x,y)>0$. By (D3), there exists $(f,\alpha)\in \mathcal{F}\times [0,+\infty)$ such that 
$$
f(D(x,y))\leq f(D(x,x_n)+D(x_n,y))+\alpha,\quad \mbox{ for all } n.
$$
On the other hand, using (D2) and ($\mathcal{F}_2$), we have
$$
\lim_{n\to +\infty}f(D(x,x_n)+D(x_n,y))+\alpha
=\lim_{n\to +\infty}f(D(x_n,x)+D(x_n,y))+\alpha=-\infty,
$$
which is a contradiction. Therefore, we have $x=y$.
\end{proof}

\begin{definition}
Let  $(X,D)$ be an $\mathcal{F}$-metric space. Let $\{x_n\}$ be a sequence in $X$.
\begin{itemize}
\item[(i)] We say that $\{x_n\}$ is $\mathcal{F}$-Cauchy, if 
$$
\lim_{n,m\to +\infty} D(x_n,x_m)=0.
$$
\item[(ii)] We say that $(X,D)$ is  $\mathcal{F}$-complete, if every $\mathcal{F}$-Cauchy sequence in $X$ is $\mathcal{F}$-convergent to a certain element in $X$.
\end{itemize}
\end{definition}

\begin{example}
Let $X=\mathbb{N}$, and let $D: X\times X\to [0,+\infty)$ be the mapping defined by \eqref{DDD}. It was shown in Example \ref{EXJ} that $(X,D)$ is an  $\mathcal{F}$-metric space with $f(t)=\frac{-1}{t}$, $t>0$, and $\alpha=1$. We shall prove that $(X,D)$ is a $\mathcal{F}$-complete. Let $\{x_n\}\subset X$ be an $\mathcal{F}$-Cauchy sequence. This means that
$$
\lim_{n,m\to +\infty}D(x_n,x_m)=0.
$$
Therefore, there exists $N\in \mathbb{N}$ such that 
$$
D(x_n,x_m)<\frac{1}{2},\quad n,m\geq N.
$$ 
Suppose that for some $n,m\geq N$, we have $x_n\neq x_m$. By the definition of $D$, and using the above inequality, we obtain
$$
1\leq D(x_n,x_m)=\exp(|x_n-x_m|)<\frac{1}{2},
$$
which is a contradiction. Then, we deduce that 
$$
x_n=x_N,\quad n\geq N,
$$
which implies that 
$$
\lim_{n\to +\infty} D(x_n,x_N)=0,
$$
i.e., $\{x_n\}$ is $\mathcal{F}$-convergent to $x_N$.  As consequence, $(X,D)$ is $\mathcal{F}$-complete.
\end{example}

\begin{proposition}\label{hihi}
Let $(X,D)$ be an $\mathcal{F}$-metric space. If $\{x_n\}\subset X$ is $\mathcal{F}$-convergent, then it is $\mathcal{F}$-Cauchy.
\end{proposition}

\begin{proof}
Let $(f,\alpha)\in \mathcal{F}\times [0,+\infty)$ be such that (D3) is satisfied. Let $x\in X$ be such that 
\begin{equation}\label{J1}
\lim_{n\to +\infty}D(x_n,x)=0.
\end{equation}
Let $\varepsilon>0$ be fixed. By ($\mathcal{F}_2$), we know that there exists some $\delta>0$ such that 
\begin{equation}\label{J2}
0<t<\delta \implies f(t)<f(\varepsilon)-\alpha.
\end{equation}
On the other hand, by \eqref{J1}, there exists some $N\in \mathbb{N}$ such that 
\begin{equation}\label{J3}
D(x_n,x)+D(x_m,x)<\delta,\quad n,m\geq N.
\end{equation}
Let $n,m\geq N$. We discuss two cases.\\
Case 1: If $x_m=x_n$. In this case, by (D1), we have
$$
D(x_n,x_m)=0<\varepsilon.
$$ 
Case 2: If $x_m\neq x_n$. In this case, from \eqref{J3}, we have
$$
0<D(x_n,x)+D(x_m,x)<\delta.
$$
Therefore, by \eqref{J2}, we have
$$
f(D(x_n,x)+D(x_m,x))<f(\varepsilon)-\alpha.
$$
Now, using (D3), we obtain
$$
f(D(x_n,x_m))\leq f(D(x_n,x)+D(x_m,x))+\alpha<f(\varepsilon),
$$
which implies from ($\mathcal{F}_1$) that 
$$
D(x_n,x_m)<\varepsilon.
$$
As consequence, we have
$$
D(x_n,x_m)<\varepsilon,\quad n,m\geq N,
$$
which yields
$$
\lim_{n,m\to +\infty}D(x_n,x_m)=0,
$$
i.e. $\{x_n\}$ is  $\mathcal{F}$-Cauchy.
\end{proof}

Next, we discuss the compactness on $\mathcal{F}$-metric spaces.

\begin{definition}
Let $(X,D)$ be an $\mathcal{F}$-metric space. Let $A$ be a nonempty subset of $X$. We say that $A$ is $\mathcal{F}$-compact if $A$ is compact with respect to the topology $\tau_{\mathcal{F}}$ on $X$.
\end{definition}

\begin{proposition}\label{PRK}
Let $(X,D)$ be an $\mathcal{F}$-metric space. Let $A$ be a nonempty subset of $X$. Then, the following statements are equivalent:
\begin{itemize}
\item[(i)] $A$ is $\mathcal{F}$-compact.
\item[(ii)] For any sequence $\{x_n\}\subset A$, there exist a subsequence $\{x_{n_k}\}$ of $\{x_n\}$ and $x\in A$  such that 
$$
\lim_{k\to +\infty} D(x_{n_k},x)=0.
$$
\end{itemize}
\end{proposition}

\begin{proof}
Suppose that $A$ is $\mathcal{F}$-compact. It can be easy seen  that any decreasing sequence of nonempty $\mathcal{F}$-closed subsets of $A$ have a nonempty intersection. Let $\{x_n\}$ be a sequence in $A$. For every $n\in \mathbb{N}$, let 
$$
C_n=\{x_m:\, m\geq n\}.
$$
Observe that $C_{n+1}\subset C_n$, for every $n\in \mathbb{N}$, which yields $\displaystyle \{\overline{C_n}\}_{n\in \mathbb{N}}$ is a decreasing sequence of nonempty $\mathcal{F}$-closed subsets of $A$. Therefore, there is some $x$ that belongs to $\displaystyle \bigcap_{n\in \mathbb{N}} \overline{C_n}$. Next, let $\varepsilon>0$ be fixed. Since $x\in \overline{C_0}$, by Proposition \ref{ayham}, there exist    $n_0\geq 0$ and $x_{n_0}\in A$ such that $D(x_{n_0},x)<\varepsilon$. continuing this process, for any $k\in \mathbb{N}$, there exist $n_k\geq k$ and  $x_{n_k}\in A$ such that $D(x_{n_k},x)<\varepsilon$. Therefore, we have 
$$
\lim_{k\to +\infty} D(x_{n_k},x)=0.
$$
On the other hand, since $A$ is $\mathcal{F}$-compact,  then it is $\mathcal{F}$-closed, and $x\in A$. Then we proved that (i)$\implies$(ii). Conversely, suppose that (ii) is satisfied.  Let $(f,\alpha)\in \mathcal{F}\times [0,+\infty)$ be such that (D3) is satisfied. First, we claim that 
\begin{equation}\label{claim1}
\forall\, r>0,\, \exists\, (x_i)_{i=1,\cdots,n}\subset A:\, 
A\subset \bigcup_{i=1,\cdots,n}B(x_i,r).
\end{equation}
We argue by contradiction, by supposing that there exists $r>0$ such that for any finite number of elements $(x_i)_{i=1,\cdots,n}\subset A$, we have 
$$
A\not\subset \bigcup_{i=1,\cdots,n}B(x_i,r).
$$ 
Let $x_1\in A$ be an arbitrary element. Then
$$
A\not\subset B(x_1,r),
$$
i.e. there exists $x_2\in A$ such that 
$$
D(x_1,x_2)\geq r.
$$ 
Again, we have
$$
A\not\subset B(x_1,r)\cup B(x_2,r),
$$
i.e.
there exists $x_3\in A$ such that 
$$
D(x_i,x_3)\geq r,\quad i=1,2.
$$
Continuing this process, by induction, we can construct a sequence $\{x_n\}\subset A$ such that 
$$
D(x_n,x_m)\geq r,\quad n,m\in \mathbb{N}^*.
$$
Observe that in this case, it is not possible to extract from $\{x_n\}$ any $\mathcal{F}$-Cauchy subsequence, so (from Proposition \ref{hihi}), any  $\mathcal{F}$-convergent subsequence. Then we obtain a contradiction with (ii), which proves \eqref{claim1}. Next, let $\{\mathcal{O}_i\}_{i\in I}$ be an arbitrary family of $\mathcal{F}$-open subsets of $X$ such that 
\begin{equation}\label{cov}
A\subset \bigcup_{i\in I} \mathcal{O}_i.
\end{equation}
We claim that
\begin{equation}\label{claim2}
\exists\, r_0>0:\, \forall\, x\in A,\, \exists\,i\in I:\, B(x,r_0)\subset \mathcal{O}_i.
\end{equation}
We argue by contradiction by supposing that for any $r>0$, there exists $x_r\in A$ such that $B(x_r,r)\not\subset \mathcal{O}_i$, for all $i\in I$. In particular, for all $n\in \mathbb{N}^*$, there exists $x_n \in A$ such that $B\left(x_n, \frac{1}{n}\right)\not\subset\mathcal{O}_i$, for all $i\in I$. By (ii), we can extract a subsequence $\{x_{n_k}\}$ from $\{x_n\}$ such that 
\begin{equation}\label{nej}
\lim_{k\to +\infty} D(x_{n_k},x)=0,
\end{equation}
for a certain $x\in A$. On the other hand, by \eqref{cov}, there exists some $j\in I$ such that $x\in \mathcal{O}_j$.
Since $\mathcal{O}_j$ is an $\mathcal{F}$-open subset of $X$, there exists some $r_0>0$ such that $B(x,r_0)\subset \mathcal{O}_j$. Next, for any $n_k\in \mathbb{N}^*$, and for any $z\in B\left(x_{n_k},\frac{1}{n_k}\right)$, we have 
$$
D(x,z)>0\implies f(D(x,z))\leq f(D(x,x_{n_k})+D(x_{n_k},z))+\alpha<f\left(D(x,x_{n_k})+\frac{1}{n_k}\right)+\alpha.
$$
By \eqref{nej} and ($\mathcal{F}_2$), there exists $K\in \mathbb{N}^*$ such that
$$
f\left(D(x,x_{n_k})+\frac{1}{n_k}\right)<f(r_0)-\alpha,\quad k\geq K,
$$
which yields
$$
D(x,z)>0\implies f(D(x,z))<f(r_0).
$$
Therefore, by ($\mathcal{F}_1$), we obtain
$$
D(x,z)<r_0.
$$
Thus we have 
$$
B\left(x_{n_k},\frac{1}{n_k}\right)\subset B(x,r_0),\quad n_k\in \mathbb{N}^*,
$$
which implies
$$
B\left(x_{n_k},\frac{1}{n_k}\right)\subset \mathcal{O}_j,\quad n_k\in \mathbb{N}^*.
$$
Observe that we obtain a contradiction with the fact that 
$B\left(x_n, \frac{1}{n}\right)\not\subset\mathcal{O}_i$, for all $i\in I$. Then \eqref{claim2} holds.  Further, by \eqref{claim1}, there exists $\displaystyle (x_p)_{p=1,\cdots,n} \subset A$ such that 
$$
A\subset \bigcup_{p=1,\cdots,n} B(x_p,r_0).
$$
But by \eqref{claim2}, for any $p=1,\cdots,n$, there exists 
$i(p)\in I$ such that $B(x_p,r_0)\subset \mathcal{O}_{i(p)}$, which yields
$$
A\subset \bigcup_{p=1,\cdots,n}\mathcal{O}_{i(p)}.
$$
Therefore, $A$ is $\mathcal{F}$-compact, and (ii)$\implies$(i).
\end{proof}

\begin{definition}
Let $(X,D)$ be an $\mathcal{F}$-metric space. Let $A$ be a nonempty subset of $X$. The subset $A$ is called sequentially $\mathcal{F}$-compact, if  for any sequence $\{x_n\}\subset A$, there exist a subsequence $\{x_{n_k}\}$ of $\{x_n\}$ and $x\in A$  such that 
$$
\lim_{k\to +\infty} D(x_{n_k},x)=0.
$$
\end{definition}

\begin{definition}
Let $(X,D)$ be an $\mathcal{F}$-metric space. Let $A$ be a nonempty subset of $X$. The subset $A$ is called  $\mathcal{F}$-totally bounded, if  
$$
\forall\, r>0,\, \exists\, (x_i)_{i=1,\cdots,n}\subset A:\, 
A\subset \bigcup_{i=1,\cdots,n}B(x_i,r).
$$
\end{definition}

From the proof of Proposition \ref{PRK}, we deduce the following result.

\begin{proposition}
Let $(X,D)$ be an $\mathcal{F}$-metric space. Let $A$ be a nonempty subset of $X$.  Then
\begin{itemize}
\item[(i)] $A$ is $\mathcal{F}$-compact if and only if $A$ is 
sequentially $\mathcal{F}$-compact.
\item[(ii)] If $A$ if $\mathcal{F}$-compact, then $A$ is $\mathcal{F}$-totally bounded.
\end{itemize}
\end{proposition}

\section{Banach contraction principle on $\mathcal{F}$-metric spaces}\label{sec5}

In this section, we establish a new version of Banach contraction principle on the setting of $\mathcal{F}$-metric spaces.

\begin{theorem}\label{TB}
Let $(X,D)$ be an $\mathcal{F}$-metric space, and let $g: X\to X$ be a given mapping. Suppose that the following conditions are satisfied:
\begin{itemize}
\item[(i)] $(X,D)$ is $\mathcal{F}$-complete.
\item[(ii)] There exists $k\in (0,1)$ such that 
$$
D(g(x),g(y))\leq k D(x,y),\quad (x,y)\in X\times X.
$$
\end{itemize}
Then $g$ has a unique fixed point $x^*\in X$. Moreover, for any $x_0\in X$, the sequence $\{x_n\}\subset X$ defined by
\begin{equation}\label{it}
x_{n+1}=g(x_n),\quad n\in \mathbb{N},
\end{equation}
is $\mathcal{F}$-convergent to $x^*$.
\end{theorem}

\begin{proof}
First, observe that $g$ has at most one fixed point. Indeed, if $(u,v)\in X\times X$ are two fixed points of $g$ with $u\neq v$, i.e.
$$
D(u,v)>0,\,\, g(u)=u,\,\, g(v)=v,
$$
then from (ii), we have
$$
D(u,v)=D(g(u),g(v))\leq k D(u,v)<D(u,v),
$$
which is a contradiction.

Next, let $(f,\alpha)\in \mathcal{F}\times [0,+\infty)$ be such that (D3) is satisfied. Let $\varepsilon>0$ be fixed.
By ($\mathcal{F}_2$), there exists $\delta>0$ such that 
\begin{equation}\label{dpu}
0<t<\delta \implies f(t)<f(\varepsilon)-\alpha.
\end{equation}
Let $x_0\in X$ be an arbitrary element.  Let $\{x_n\}\subset X$ be the sequence defined by \eqref{it}.  Without restriction of the generality, we may suppose that $D(x_0,x_1)>0$. Otherwise, $x_0$ will be a fixed point of $g$. It can be easily seen that from (ii), we have
$$
D(x_n,x_{n+1})\leq k^n D(x_0,x_1),\quad n\in \mathbb{N},
$$
which yields
$$
\sum_{i=n}^{m-1} D(x_i,x_{i+1})\leq  \frac{k^n}{1-k} D(x_0,x_1),\quad m>n.
$$ 
Since
$$
\lim_{n\to +\infty}\frac{k^n}{1-k} D(x_0,x_1)=0,
$$
there exists some $N\in \mathbb{N}$ such that 
\begin{equation}\label{oo}
0<\frac{k^n}{1-k} D(x_0,x_1)<\delta,\quad n\geq N.
\end{equation}
Hence, by\eqref{dpu} and $(\mathcal{F}_1$), we have
\begin{equation}\label{DB}
f\left(\sum_{i=n}^{m-1} D(x_i,x_{i+1})\right)\leq f\left(\frac{k^n}{1-k} D(x_0,x_1)\right)<f(\varepsilon)-\alpha,\quad m>n\geq N.
\end{equation}
Using (D3) and \eqref{DB}, we obtain
$$
D(x_n,x_m)>0,\, m>n\geq N \implies f(D(x_n,x_m))\leq f\left(\sum_{i=n}^{m-1} D(x_i,x_{i+1})\right)+\alpha<f(\varepsilon),
$$
which implies by  ($\mathcal{F}_1$) that
$$
D(x_n,x_m)<\varepsilon,\quad m>n\geq N.
$$
This proves that $\{x_n\}$ is $\mathcal{F}$-Cauchy. Since $(X,D)$ is $\mathcal{F}$-complete, there exists $x^*\in X$ such that $\{x_n\}$ is $\mathcal{F}$-convergent to $x^*$, i.e.
\begin{equation}\label{mur}
\lim_{n\to +\infty} D(x_n,x^*)=0.
\end{equation}
We shall prove that $x^*$ is a fixed point of $g$. We argue by contradiction by supposing that $D(g(x^*),x^*)>0$. By (D3), we have
$$
f(D(g(x^*),x^*))\leq f(D(g(x^*),g(x_n))+D(g(x_n),x^*))+\alpha,\quad n\in \mathbb{N}.
$$
Using (ii) and ($\mathcal{F}_1$), we obtain
$$
f(D(g(x^*),x^*))\leq f(kD(x^*,x_n)+D(x_{n+1},x^*))+\alpha,\quad n\in \mathbb{N}.
$$
On the other hand, using ($\mathcal{F}_2$) and \eqref{mur}, we have
$$
\lim_{n\to +\infty} f(kD(x^*,x_n)+D(x_{n+1},x^*))+\alpha=-\infty,
$$
which is a contradiction. Therefore, we have $D(g(x^*),x^*)=0$, i.e. $g(x^*)=x^*$. As consequence, $x^*\in X$ is the unique fixed point of $g$.
\end{proof}

\begin{corollary}
Let $(X,D)$ be an  $\mathcal{F}$-metric space, and $(f,\alpha)\in \mathcal{F}\times [0,+\infty)$ be such that (D3) is satisfied. Let $g: B(x_0,r)\to X$ be a given mapping, where $x_0\in X$ and $r>0$. Suppose that the following conditions are satisfied:
\begin{itemize}
\item[(i)] \eqref{JAL} is satisfied.
\item[(ii)] $(X,D)$ is $\mathcal{F}$-complete.
\item[(iii)] There exists $k\in (0,1)$ such that 
$$
D(g(x),g(y))\leq k D(x,y),\quad (x,y)\in B(x_0,r)\times B(x_0,r).
$$
\item[(iv)] There exists $0<\varepsilon<r$ such that 
$$
f\left(k \varepsilon+D(x_0,g(x_0))\right)\leq f(\varepsilon)-\alpha.
$$
\end{itemize}
Then $g$ has a fixed point.
\end{corollary}

\begin{proof}
Let $0<\varepsilon<r$ be such that (iv) is satisfied.  
First, we shall prove that 
\begin{equation}\label{st}
g(\mathbf{B}(x_0,\varepsilon))\subset \mathbf{B}(x_0,\varepsilon).
\end{equation}
Let $x\in \mathbf{B}(x_0,\varepsilon)$, i.e.
$$
D(x_0,x)\leq \varepsilon.
$$
Suppose that $D(g(x),x_0)>0$. By (D3), we have
$$
f(D(g(x),x_0))\leq f(D(g(x),g(x_0))+D(g(x_0),x_0))+\alpha.
$$
Using ($\mathcal{F}_1$), (iii), and (iv), we obtain
\begin{eqnarray*}
f(D(g(x),x_0)) &\leq & f(k D(x,x_0)+D(g(x_0),x_0))+\alpha\\
&\leq & f(k\varepsilon +D(g(x_0),x_0))+\alpha\\
&\leq & f(\varepsilon).
\end{eqnarray*}
Hence, by ($\mathcal{F}_1$), we have $D(g(x),x_0)\leq \varepsilon$, which yields $g(x)\in \mathbf{B}(x_0,\varepsilon)$. Therefore, we proved \eqref{st}. Further, the mapping $g: \mathbf{B}(x_0,\varepsilon)\to \mathbf{B}(x_0,\varepsilon)$ is well-defined, and satisfies the Banach contraction. On the other hand, since \eqref{JAL} is satisfied, by Proposition \ref{CL}, we know that $\mathbf{B}(x_0,\varepsilon)$  is $\mathcal{F}$-closed, so from (i), it is $\mathcal{F}$-complete. Finally the result follows from Theorem \ref{TB}.
\end{proof}

\vspace{1cm}

\noindent {\bf Acknowledgements.} The second author extends his appreciation to Distinguished Scientist Fellowship Program (DSFP) at King Saud University (Saudi Arabia).

\vspace{1cm}

\noindent Mohamed Jleli\\
Department of Mathematics, College of Science, King Saud University, Riyadh 11451, Saudi Arabia.\\
E-mail: jleli@ksu.edu.sa\\

\noindent Bessem Samet\\
Department of Mathematics, College of Science, King Saud University, Riyadh 11451, Saudi Arabia.\\
E-mail: bsamet@ksu.edu.sa

\end{document}